\documentclass[10pt, reqno]{amsart}
\numberwithin{equation}{section}
\usepackage{amssymb, color, cite}
\usepackage{hyperref}

\newcommand{\qtq}[1]{\quad\text{#1}\quad}

\newcommand{\R}{\mathbb{R}}
\newcommand{\C}{\mathbb{C}}

\newcommand{\eps}{\varepsilon}

\newtheorem{theorem}{Theorem}[section]

\newtheorem{lemma}[theorem]{Lemma}

\newtheorem{corollary}[theorem]{Corollary}

\newtheorem{proposition}[theorem]{Proposition}

\theoremstyle{definition}

\theoremstyle{remark}

\begin{document}

\title[NLS Recovery]{Recovery of a spatially-dependent coefficient from the NLS scattering map}
\author{Jason Murphy}
\address{Department of Mathematics \& Statistics, Missouri S\&T}
\email{jason.murphy@mst.edu}

\begin{abstract} We follow up on work of Strauss, Weder, and Watanabe concerning scattering and inverse scattering for nonlinear Schr\"odinger equations with nonlinearities of the form $\alpha(x)|u|^p u$. 
\end{abstract}

\maketitle

\section{Introduction}

This note is intended to follow up on some previous works \cite{Strauss, Weder1, Weder3, Watanabe} concerning nonlinear Schr\"odinger equations of the form
\begin{equation}\label{nls}
(i\partial_t + \Delta)u = \alpha(x)|u|^p u,\quad (t,x)\in\R\times\R^d. 
\end{equation}
These works considered the problems of (i) scattering for a suitable class of data and (ii) the determination of the nonlinearity from knowledge of the scattering map. 

In \cite{Strauss}, Strauss established a small-data scattering theory for \eqref{nls} in $H^s$, with $p$ an integer in the mass-supercritical regime (i.e. $p>\tfrac{4}{d}$), $s$ sufficiently large, and $\alpha\in W^{s,\infty}$.  The need for high regularity was essentially a consequence of estimating solutions using the $L^\infty$-norm, with the nonlinear term in the Duhamel formula being estimated directly via the dispersive estimate.  After establishing the small-data scattering theory, Strauss further demonstrated that knowledge of the scattering map suffices to determine integrals of the form
\[
\int_\R \langle \alpha |e^{it\Delta}\varphi|^p e^{it\Delta}\varphi,e^{it\Delta}\psi\rangle\,dt,
\]
which may be used to recover the coefficient $\alpha$ pointwise.  The result of \cite{Strauss} was extended in works of Weder \cite{Weder1, Weder3}, who considered equations of the form
\[
(i\partial_t + \Delta)u = V_0(x) u + \sum_{k=k_0}^\infty V_k(x) |u|^{2k}u
\]
and used the small-data scattering map to determine the functions $V_k$, including the potential $V_0$. The constant $k_0$ was chosen so that the lowest power in the nonlinearity exceeded the so-called \emph{Strauss exponent} (allowing for some mass-subcritical nonlinearities); scattering was obtained in $H^s$ for some integer $s>\tfrac{d}{2}-1$ ($s=1$ in $d=1$); and the coefficients were assumed to satisfy $V_k\in W^{s,\infty}$.  Weder also relied primarily on dispersive estimates (of the type obtained in \cite{Weder6}) to estimate the nonlinear terms.

In \cite{Watanabe}, Watanabe established a large-data $H^1$ scattering theory for \eqref{nls} in the $3d$ intercritical regime ($\tfrac{4}{d}<p<\tfrac{4}{d-2}$) for \eqref{nls} with decaying coefficients $\alpha$ satisfying a repulsivity condition.  He then adapted techniques from \cite{EW} to the setting of \eqref{nls}, evaluating the scattering map on data of the form $e^{i\rho\theta\cdot x}\varphi$ with $\rho\gg 1$ to determine integrals of the form 
\[
\int_\R \langle \alpha(\cdot+2t\theta)|\varphi|^p \varphi,\psi\rangle\,dt,
\]
which then determine the X-ray transform of $\alpha$. 

Our first contribution is to revisit the approach of \cite{Strauss, Weder1, Weder3} and to lower the regularity assumptions by utilizing Strichartz estimates instead of directly using the dispersive estimate.  This is similar to the approach taken in the related work \cite{CarlesGallagher}, although in this latter work the authors were primarily concerned with the analyticity of the scattering operator, and correspondingly the results concerning NLS were restricted to the case $p\in 2\mathbb{N}$ and $\alpha$ constant.
We further extend the work of \cite{Strauss, Weder1, Weder6} by establishing analogous results in the full mass-subcritical regime.

Our second contribution is to extend the results of \cite{Watanabe} to the mass-critical and mass-subcritical regime in dimensions $d\geq 3$.  We follow essentially the same strategy to recover $\alpha$ from the scattering map.  In contrast to \cite{Watanabe}, however, we formulate the original scattering problem as a small-data problem in a suitable weighted space.  This construction directly provides us with the key estimate needed to control the nonlinear error term in the reconstruction argument.  The formulation as a small-data problem also removes the need for any sign or repulsivity conditions on the coefficient.  After presenting our approach, we will also discuss some challenges associated to this problem in the mass-supercritical regime. 

Our main results appear below as follows:
\begin{itemize}
\item Theorem~\ref{T:intercritical} -  small-data scattering in $H^1$ in the intercritical case;
\item Theorem~\ref{T:subcritical} - small-data scattering in $L^2$ in the mass-critical and mass-subcritical case; 
\item Theorem~\ref{T:subcritical2} - scattering in $L^2$ in the mass-critical and mass-subcritical case with boosted data;
\item Theorem~\ref{T:recovery1} and Corollary~\ref{C1} - recovery of the nonlinearity from the scattering map in the setting of Theorem~\ref{T:intercritical} and Theorem~\ref{T:subcritical}; 
\item Theorem~\ref{T:recovery2} and Corollary~\ref{C2} - recovery of the nonlinearity from the scattering map in the setting of Theorem~\ref{T:subcritical2}. 
\end{itemize}

Our results fit in the broader context of the recovery of the nonlinear terms from scattering data for nonlinear dispersive equations.  For some further results of this type (primarily in the NLS setting), we refer the reader to \cite{SBUW, CarlesGallagher, MorStr, Weder0, Weder6, Weder3, Weder4, Weder1, Sasaki2, Sasaki, KMV}.  We also mention the related works \cite{SBS2, HMG}, which considered the recovery of spatially-dependent coefficients in the nonlinearity using particular solutions rather than the scattering map.


\subsection*{Acknowledgments} The author was supported in part by NSF DMS-2137217.  I am grateful to Rowan Killip, Monica Visan, Michiyuki Watanabe, and John Singler for helpful discussions related to this work. 

\section{Preliminaries}\label{S:Prelim}

We write $A\lesssim B$ to denote the inequality $A\leq CB$ for some $C>0$.  We denote dependence on parameters by subscripts, e.g. $A\lesssim_\ell B$ means $A\leq CB$ for some $C=C(\ell)>0$.  We utilize the standard space-time Lebesgue spaces, i.e.
\[
\|u\|_{L_t^q L_x^r(I\times\R^d)} = \bigl\| \|u(t,\cdot)\|_{L_x^r(\R^d)}\bigr\|_{L_t^q(I)}. 
\]
We use $H_x^{1,r}$ for the Sobolev space with norm
\[
\|u\|_{H^{1,r}}=\|u\|_{L^r}+\|\nabla u\|_{L^r}. 
\]
We write $q'$ for the H\"older dual of $q$.  The Fourier multiplier operator with symbol $m$ is denoted by $m(i\nabla)$.  Finally, we write $\langle x\rangle=\sqrt{1+|x|^2}$. 

The free Schr\"odinger group is denoted $e^{it\Delta}$.  We have the following identity for boosted initial data: for $v\in\R^d$, 
\begin{equation}\label{boost}
[e^{it\Delta}e^{iv\cdot x}\varphi](x) = e^{-i|v|^2 t}e^{iv\cdot x}[e^{it\Delta}\varphi](x-2tv).
\end{equation}
The Schr\"odinger group also obeys the following dispersive estimates
\[
\|e^{it\Delta}\varphi\|_{L^\infty} \lesssim |t|^{-\frac{d}{2}} \|\varphi\|_{L^1},\quad \| e^{it\Delta}\varphi\|_{L^2} = \|\varphi\|_{L^2},
\]
which (by interpolation) yield the following (Lorentz-improved) estimates
\begin{equation}\label{disp-lorentz}
\|e^{it\Delta} \varphi\|_{L^{r,2}} \lesssim |t|^{-(\frac{d}{2}-\frac{d}{r})}\|\varphi\|_{L^{r',2}},\quad 2\leq r<\infty.
\end{equation}

We will also make use of the standard Strichartz estimates for $e^{it\Delta}$.  We call a pair $(q,r)$ \emph{admissible} if $2\leq q,r\leq\infty$, $\tfrac{2}{q}+\tfrac{d}{r}=\tfrac{d}{2}$, and $(q,r,d)\neq(2,\infty,2)$.

\begin{theorem}[Strichartz estimates, \cite{GV, KT, Strichartz}]  For any admissible $(q,r)$ and any $\varphi\in L^2$, we have
\[
\|e^{it\Delta}\varphi\|_{L_t^q L_x^r(\R\times\R^d)}\lesssim \|\varphi\|_{L^2}.
\]
For any admissible $(q,r)$ and $(\tilde q,\tilde r)$ and $F\in L_t^{\tilde q'}L_x^{\tilde r'}(\R\times\R^d)$, we have
\[
\biggl\|\int_{-\infty}^t e^{i(t-s)\Delta}F(s)\,ds\biggr\|_{L_t^q L_x^r(\R\times\R^d)}\lesssim \|F\|_{L_t^{\tilde q'} L_x^{\tilde r'}(\R\times\R^d)}. 
\]
\end{theorem}

\subsection{Weighted estimate for boosted data} The following estimate concerning boosted solutions to the linear Schr\"odinger equation will play a key role in Theorem~\ref{T:subcritical2} and Theorem~\ref{T:recovery2} below.  The estimate is modeled closely after estimates appearing in \cite{Enss, EW, Watanabe}.  

Given $s\in[0,\tfrac{d}{2})$, we introduce the space $X^{s}(\R^d)$ via the norm
\begin{equation}\label{Xs}
\|\varphi\|_{X^s} = \|\langle x\rangle^s \varphi\|_{L^2} + \| |\nabla|^s \varphi\|_{L^{\frac{2d}{d+2s}}}. 
\end{equation}

\begin{proposition}\label{P:Enss} Let $q:\R^d\to\C$ satisfy $|q(x)|\lesssim\langle x\rangle^{-s}$ for some $s\in(0,\tfrac{d}{2})$.  Then
\[
\| q\,e^{it\Delta} e^{iv\cdot x}\varphi\|_{L^2} \lesssim \langle tv\rangle^{-s}\|\varphi\|_{X^{s}}\qtq{uniformly in}t\in\R.
\]
\end{proposition}
 
We begin with a mismatch-type estimate (also found in \cite{Enss, EW}).  

\begin{lemma}\label{L:Enss} Let $g\in C_c^\infty(\R^d)$ satisfy $\text{supp\,}g\subset\{|\xi|\leq N\}$ for some $N\geq 1$. Let $t\in\R$ and suppose $S,S'\subset\R^d$ are measurable sets satisfying 
\begin{equation}\label{dist}
\text{dist}(S,S')=R\geq 4N|t|.
\end{equation}
Then for any $\ell\geq 0$, 
\[
\|\chi_{S'} e^{it\Delta}g(i\nabla)\chi_S\|_{L^2\to L^2} \lesssim_{\ell,g} (1+R)^{-\ell}.
\]
The estimate is uniform in $t\in\R$. 
\end{lemma}

\begin{proof} We begin by observing that for a bounded continuous function $m$, we have
\[
\|\chi_{S'}m(i\nabla)\chi_S\|_{L^2\to L^2}\lesssim\min\{\|m\|_{L^\infty},\|\check m\|_{L^1(|x|>R)}\}. 
\]
The $L^\infty$ bound follows from Plancherel.  For the remaining estimate, one may proceed by expanding the definition of
\[
\|\chi_{S'}m(i\nabla)\chi_S f\|_{L^2}^2,\quad f\in L^2,
\]
and using the inequality 
\[
|f(y)|\,|f(z)|\leq\tfrac12[|f(y)|^2+|f(z)|^2]
\]
along with the assumption \eqref{dist} (see \cite[Lemma~2.1]{Enss}).  In the present setting, we take $m(\xi)=e^{-it|\xi|^2}g(\xi)$ and seek to estimate
\[
\check m(x)=\int e^{ix\xi-it|\xi|^2}g(\xi)\,d\xi,\quad |x|>R. 
\]
According to our assumptions, the phase has no stationary points, and hence repeated integration by parts leads to bounds of the form $C_\ell|x-2t\xi|^{-\ell}$ for arbitrary $\ell$.  In the present setting, we have
\[
|x-2t\xi|\geq |x|-2t|\xi|\geq \tfrac12|x|\geq\tfrac 12 R,
\]
and hence we obtain 
\[
\|\check m\|_{L^1(|x|>R)}\lesssim_\ell R^{-\ell}
\]
for any $\ell\geq 0$. 
\end{proof}

\begin{proof}[Proof of Proposition~\ref{P:Enss}] We let $v\in\R^d\backslash\{0\}$ and $s\in(0,\tfrac{d}{2})$.  By \eqref{boost}, it is enough to show that
\begin{equation}\label{PETS}
\|q(\cdot+2tv)e^{it\Delta}\varphi\|_{L^2} \lesssim_s \langle t v\rangle^{-s}\|\varphi\|_{X^{s}}.
\end{equation}

We split $\varphi$ into low and high frequencies via 
\[
\varphi=P_{\leq N}\varphi+P_{>N}\varphi,\quad N:=\tfrac14|v|. 
\]

We first estimate the low frequencies.  We set 
\[
S=\{|x|\leq\tfrac{1}{10}|tv|\}
\]
and use the triangle inequality to obtain
\begin{align}
\|q(\cdot+2tv)e^{it\Delta}P_{\leq N}\varphi\|_{L^2} & \leq
\|q(\cdot+2tv)e^{it\Delta}P_{\leq N}[1-\chi_S]\varphi\|_{L^2} \label{Eest1}\\
&\quad + \|[1-\chi_S(\cdot+2tv)]q(\cdot+2tv)e^{it\Delta}P_{\leq N}\chi_S\varphi\|_{L^2}\label{Eest2}\\
&\quad +  \|\chi_S(\cdot + 2tv)q(\cdot+2tv)e^{it\Delta}P_{\leq N}\chi_S\varphi\|_{L^2}.\label{Eest3} 
\end{align}

For \eqref{Eest1}, we estimate
\begin{align*}
\|q(\cdot+2tv)e^{it\Delta}P_{\leq N}[1-\chi_S]\varphi\|_{L^2} & \lesssim \|q\|_{L^\infty}\|\varphi\|_{L^2(|x|>\frac{1}{10}|tv|)} \\
&  \lesssim \langle tv\rangle^{-s}\|\langle x\rangle^s \varphi\|_{L^2},
\end{align*}
which is acceptable. 

For \eqref{Eest2}, we use the decay assumption on $q$ to obtain
\begin{align*}
\|[1-\chi_S(\cdot+2tv)]q(\cdot+2tv)e^{it\Delta}P_{\leq N}\chi_S\varphi\|_{L^2} & \lesssim  \langle tv\rangle^{-s}\|\varphi\|_{L^2},
\end{align*}
which is acceptable.

For \eqref{Eest3} we introduce 
\[
S'=\{|x+2tv|\leq\tfrac{1}{10}|tv|\}
\]
and observe that
\[
\text{dist}(S,S')\geq |v|t=4N|t|.
\]
Thus, Lemma~\ref{L:Enss} implies that
\begin{align*}
\|\chi_S(\cdot + 2tv)q(\cdot+2tv)e^{it\Delta}P_{\leq N}\chi_S\varphi\|_{L^2}& \lesssim \|q\|_{L^\infty} \|\chi_{S'} e^{it\Delta}P_{\leq N} \chi_S \varphi\|_{L^2} \\
& \lesssim \langle tv\rangle^{-s}\|\varphi\|_{L^2},
\end{align*}
which is acceptable.

It remains to estimate the high frequencies.  As it is straightforward to obtain the bound
\[
\|q(\cdot+2tv)e^{it\Delta}P_{>N}\varphi\|_{L^2} \lesssim \|q\|_{L^\infty}\|\varphi\|_{L^2}\lesssim\|\varphi\|_{L^2},
\]
it suffices to obtain the $|tv|^{-s}$ bound.  To this end, we use H\"older's inequality (in Lorentz spaces), the dispersive estimate \eqref{disp-lorentz}, the embedding $L^r\hookrightarrow L^{r,2}$ for $r\leq 2$, and Bernstein's inequality (recalling $|v|=4N$) to obtain
\begin{align*}
\|q(\cdot+2tv)e^{it\Delta}P_{>N}\varphi\|_{L^2} & \lesssim \|\langle x\rangle^{-s}\|_{L^{\frac{d}{s},\infty}} \|e^{it\Delta}P_{>N}\varphi\|_{L^{\frac{2d}{d-2s},2}} \\
& \lesssim |t|^{-s}\|P_{>N}\varphi\|_{L^{\frac{2d}{d+2s},2}} \\
& \lesssim |tv|^{-s} \| |\nabla|^s \varphi\|_{L^{\frac{2d}{d+2s}}},
\end{align*}
which is acceptable. \end{proof}


\section{The Direct Problem}

In this section we prove several scattering results for \eqref{nls}.  We first establish scattering for small data in Sobolev spaces.  We utilize standard contraction mapping arguments based on Strichartz estimates (see e.g. \cite{CazenaveWeissler}).  In the intercritical regime ($\tfrac{4}{d}\leq p\leq \tfrac{4}{d-2}$), the coefficient $\alpha$ and its gradient are estimated in $L^\infty$.  In the mass-subcritical regime $(p<\tfrac{4}{d})$ we impose a decay assumption on $\alpha$. 

Throughout the rest of the paper, we will regularly make use of the admissible pair
\begin{equation}\label{qr}
(q,r) = (p+2,\tfrac{2d(p+2)}{d(p+2)-4})
\end{equation}
(note that we will restrict to $p\geq 2$ in dimension $d=1$).

\begin{theorem}\label{T:intercritical} Let $d\geq 1$ and suppose $p$ satisfies
\[
\begin{cases} \tfrac{4}{d}\leq p\leq \tfrac{4}{d-2} & d\geq 3, \\ \tfrac{4}{d}\leq p<\infty & d\in\{1,2\}.\end{cases} 
\]
Let $\alpha$ be a continuous function with $\alpha,\nabla\alpha \in L^\infty$.  There exists $\eta>0$ sufficiently small that for any $u_-\in H^1$ with $\|u_-\|_{H^1}<\eta$, there exists a unique global solution $u$ to \eqref{nls} and final state $u_+\in H^1$ satisfying
\begin{equation}\label{STB1}
\| u\|_{L_t^q H_x^{1,r}(\R\times\R^d)} \lesssim \|u_-\|_{H^1} 
\qtq{and}
\lim_{t\to\pm\infty} \|u(t)-e^{it\Delta}u_{\pm}\|_{H^1} = 0,
\end{equation}
where $(q,r)$ is as in \eqref{qr}. 
\end{theorem}

\begin{proof} Let $u_-\in H^1$.  We will prove that if $\|u_-\|_{H^1}$ is sufficiently small, the map
\begin{equation}\label{Phi}
u\mapsto \Phi(u) = e^{it\Delta}u_- -i\int_{-\infty}^t e^{i(t-s)\Delta}\alpha|u(s)|^p u(s)\,ds
\end{equation}
is a contraction on a suitable metric space. To this end, we define
\[
X=\{u:\R\times\R^d\to\C\ |\ \|u\|_{L_t^q H_x^{1,r}(\R\times\R^d)} \leq 2C\|u_-\|_{H^1}\},
\]
which we equip with the metric
\[
d(u,v) = \|u-v\|_{L_t^q L_x^r(\R\times\R^d)}. 
\]
The constant $C$ encodes implicit constants appearing in estimates such as Strichartz and Sobolev embedding.  Throughout the proof, all space-time norms will be taken over $\R\times\R^d$ unless indicated otherwise.

We define $r_c = \tfrac{dp(p+2)}{4}$ and observe that by Sobolev embedding
\[
\|u\|_{L_x^{r_c}} \lesssim \| |\nabla|^{s_c} u\|_{L_x^r} \lesssim \| u\|_{H_x^{1,r}},\qtq{where} s_c=\tfrac{d}{2}-\tfrac{2}{p}\in[0,1].
\]

Now let $u\in X$.  By Strichartz, H\"older, the chain and product rules, we have
\begin{align*}
\|\Phi(u)\|_{L_t^q H_x^{1,r}} & \lesssim \|u_-\|_{H^1} + \|[\alpha |u|^p u]\|_{L_t^{q'} H_x^{1,r'}} \\
& \lesssim \|u_-\|_{H^1} + (\|\alpha\|_{L^\infty}+\|\nabla\alpha\|_{L^\infty}) \|u\|_{L_t^q L_x^{r_c}}^p \| u\|_{L_t^q H_x^{1,r}} \\
& \lesssim \|u_-\|_{H^1} + \| u\|_{L_t^q H_x^{1,r}}^{p+1}  \lesssim \|u_-\|_{H^1} + \|u_-\|_{H^1}^{p+1}.
\end{align*}
It follows that for $\|u_-\|_{H^1}$ sufficiently small, $\Phi:X\to X$. 

Given $u,v\in X$, we similarly estimate
\begin{align*}
\|u-v\|_{L_t^q L_x^r} & \lesssim \|\alpha[|u|^p u-|v|^p v]\|_{L_t^{q'}L_x^{r'}} \\
& \lesssim \|\alpha\|_{L^\infty}\{\|u\|_{L_t^q L_x^{r_c}}^p + \|v\|_{L_t^q L_x^{r_c}}^p\}\|u-v\|_{L_t^q L_x^r} \\
& \lesssim \|u_-\|_{H^1}^p \|u-v\|_{L_t^q L_x^r},
\end{align*}
which shows that $\Phi$ is a contraction provided $\|u_-\|_{H^1}$ is sufficiently small. 

We conclude that $\Phi$ has a unique fixed point $u\in X$, yielding the desired solution to \eqref{nls}.  The convergence $e^{-it\Delta}u(t)\to u_-$ in $H^1$ as $t\to-\infty$ follows by construction and the estimates above.  To prove the existence of a scattering state as $t\to\infty$, we estimate as above to obtain
\begin{align*}
\|e^{-it\Delta}u(t)-e^{-is\Delta}u(s)\|_{H^1} & \lesssim (\|\alpha\|_{L^\infty}+\|\nabla\alpha\|_{L^\infty}) \| u\|_{L_t^q H_x^{1,r}((s,t)\times\R^d)}^{p+1}\\
&\to 0 \qtq{as}s,t\to\infty. 
\end{align*}
Thus $\{e^{-it\Delta}u(t)\}$ is Cauchy and so converges to a unique $u_+\in H^1$ as $t\to\infty$. 
\end{proof}

We next consider the mass-subcritical regime. Assuming that $\alpha$ belongs to a suitable Lebesgue space, we can first establish scattering for small $L^2$ data. In fact, this result (as well as Theorem~\ref{T:subcritical2} below) allows for $p$ to go below the usual long-range exponent $p=\tfrac{2}{d}$.

\begin{theorem}\label{T:subcritical} Let $d\geq 1$ and suppose $p$ satisfies
\[
\begin{cases}
0<p\leq \tfrac{4}{d} & d\geq 2, \\
2 \leq p \leq4 & d=1.
\end{cases}
\]
Let $\alpha$ be a continuous function with $\alpha \in L^\infty\cap L^{\frac{2d}{4-dp}}$.  There exists $\eta>0$ sufficiently small that for any $u_-\in L^2$ with $\|u_-\|_{L^2}<\eta$, there exists a unique global solution $u$ to \eqref{nls} and final state $u_+\in L^2$ satisfying
\begin{equation}\label{STB2}
\|u\|_{L_t^q L_x^r(\R\times\R^d)}\lesssim \|u_-\|_{L^2} \qtq{and}
\lim_{t\to\pm\infty}\|u(t)-e^{it\Delta}u_\pm\|_{L^2} = 0,
\end{equation}
where $(q,r)$ is as in \eqref{qr}. 
\end{theorem}

\begin{proof} We show that $\Phi$ defined in \eqref{Phi} is a contraction on the complete metric space
\[
X=\{u:\R\times\R^d\to\C\ |\ \|u\|_{L_t^q L_x^r(\R\times\R^d)}\leq 2C\|u_-\|_{L^2}\},
\]
with metric given by
\[
d(u,v) = \|u-v\|_{L_t^q L_x^r(\R\times\R^d)}.
\]
Once again $C$ encodes implicit constants appearing in the estimates below. 

The essential step is the following nonlinear estimate: by Strichartz and H\"older's inequality, we have 
\begin{align*}
\biggl\|\int_{-\infty}^t e^{i(t-s)\Delta}\alpha(x)|u(s)|^p u(s)\,ds\, \biggr\|_{L_t^q L_x^r}  &\lesssim \|\alpha |u|^p u\|_{L_t^{q'}L_x^{r'}} \\
& \lesssim \|\alpha\|_{L^{\frac{2d}{4-dp}}} \|u\|_{L_t^q L_x^r}^{p+1}.
\end{align*}
With this estimate in hand, the proof exactly parallels that of Theorem~\ref{T:intercritical}.  The constraint $p\geq 2$ in in $d=1$ is necessary to use the space $L_t^q L_x^r$ for $u$ (see \eqref{qr}) as well as the space $L^{\frac{2d}{4-dp}}$ for $\alpha$ 
\end{proof}

We next establish a mass-critical and mass-subcritical scattering theory for a class of data adapted to the setting of \cite{Watanabe}, namely, data of the form $u_-=e^{iv\cdot x}\varphi$ with $|v|\gg 1$. By working with a suitably weighted space (and imposing further decay assumptions on $\alpha$), we can recast the scattering problem for such data as a small-data problem.

Given $a\geq 0$ and $s\in[0,\tfrac{d}{2})$ we introduce the space $X^{a,s}$ via the norm 
\begin{equation}\label{Xas}
\|\varphi\|_{X^{a,s}} =\||\nabla|^a\varphi\|_{L^2} + \|\varphi\|_{X^s},
\end{equation}
where $X^{s}$ is as in \eqref{Xs}. 
To simplify the formulas below, we also introduce the parameter
\begin{equation}\label{cpd}
c=c(p,d)= \tfrac{4-p(d-2)}{2(p+2)}.
\end{equation}

\begin{theorem}\label{T:subcritical2} Let $d\geq 3$ and $0<p\leq \tfrac{4}{d}$. Let $\sigma$ satisfy
\begin{equation}\label{sigma}
 \max\{\tfrac{2}{4-p(d-2)}, \tfrac{4-dp}{4-p(d-2)}\} <\sigma<\tfrac{d}{2}
\end{equation}
and suppose that $\alpha$ is a continuous function satisfying
\begin{equation}\label{alphaW}
\langle x\rangle^{(p+1)c\sigma}\alpha \in L^{\frac{2d(p+2)}{4-dp}}. 
\end{equation}
Let $\varphi\in X^{1,\sigma}$. For $|v|$ sufficiently large, there exists a unique global solution to \eqref{nls} and final state $u_+\in L^2$ satisfying
\begin{equation}\label{WSTB}
\|\langle x\rangle^{-c\sigma}u\|_{L_{t,x}^{p+2}(\R\times\R^d)} \lesssim |v|^{-\frac{1}{p+2}}\|\varphi\|_{X^{1,\sigma}}
\end{equation}
and
\[
\lim_{t\to\pm\infty}\|u(t)-e^{it\Delta}u_\pm\|_{L^2} = 0, \qtq{where}u_- = e^{iv\cdot x}\varphi.
\]
\end{theorem}

%

\begin{proof} We wish to close a contraction mapping argument for the map $\Phi$ in \eqref{Phi} in the space
\[
Y=\{u:\R\times\R^d\to\C \ | \ \|\langle x\rangle^{-c\sigma}u\|_{L_{t,x}^{p+2}}\leq 2C|v|^{-\frac{1}{p+2}}\|\varphi\|_{X^{1,\sigma}}\}
\]
with metric
\[
d(u,v)=\|\langle x\rangle^{-c\sigma}[u-v]\|_{L_{t,x}^{p+2}}.
\]
The constant $C$ encodes implicit constants appearing in several inequalities, including the Strichartz estimates and the inequality in Proposition~\ref{P:Enss}.

We begin with the linear term in the definition of $\Phi$ (see \eqref{Phi}).  By H\"older's inequality, we have
\begin{align*}
\| \langle x\rangle^{-c\sigma}e^{it\Delta}e^{iv\cdot x}\varphi\|_{L_x^{p+2}} \lesssim \|\langle x\rangle^{-\sigma}e^{it\Delta}e^{iv\cdot x}\varphi\|_{L_x^2}^{c}\|e^{it\Delta}e^{iv\cdot x}\varphi\|_{L_x^{\frac{2d}{d-2}}}^{1-c}.
\end{align*}
Using \eqref{boost} and Sobolev embedding, we obtain
\[
\|e^{it\Delta}e^{iv\cdot x}\varphi\|_{L^{\frac{2d}{d-2}}}\lesssim \|e^{it\Delta}\varphi\|_{L^{\frac{2d}{d-2}}} \lesssim \|\varphi\|_{\dot H^1},
\]
while Proposition~\ref{P:Enss} implies
\[
\|\langle x\rangle^{-\sigma}e^{it\Delta}e^{iv\cdot x}\varphi\|_{L^2}  \lesssim \langle vt\rangle^{-\sigma}\|\varphi\|_{X^{\sigma}}.
\]
Thus, by \eqref{sigma} and a change of variables, we obtain 
\begin{align*}
\| \langle x\rangle^{-c\sigma}e^{it\Delta}e^{iv\cdot x}\varphi\|_{L_{t,x}^{p+2}} & \lesssim \| \langle vt\rangle^{-c\sigma}\|_{L_t^{p+2}} \|\varphi\|_{X^{1,\sigma}} \lesssim |v|^{-\frac{1}{p+2}} \|\varphi\|_{X^{1,\sigma}}.
\end{align*}

We turn to the nonlinear estimate.  We will use the same Strichartz pair $(q,r)$ as in the proofs of Theorem~\ref{T:intercritical} and Theorem~\ref{T:subcritical}; see \eqref{qr}.  We also observe that \eqref{sigma} guarantees
\[
\langle x\rangle^{-c\sigma}\in L^{\frac{2d(p+2)}{4-dp}}. 
\]
Thus, given $u\in Y$, we use H\"older's inequality, Strichartz, and \eqref{alphaW} to obtain
\begin{align*}
\biggl\|& \langle x\rangle^{-c\sigma} \int_{-\infty}^t e^{i(t-s)\Delta}\alpha|u|^p u(s)\,ds\biggr\|_{L_{t,x}^{p+2}} \\
& \lesssim \|\langle x\rangle^{-c\sigma}\|_{L^{\frac{2d(p+2)}{4-dp}}}\biggl\|\int_{-\infty}^t e^{i(t-s)\Delta}\alpha(x)|u|^p u(s)\,ds\biggr\|_{L_t^q L_x^r} \\
& \lesssim \| \alpha |u|^p u\|_{L_t^{q'}L_x^{r'}} \\
& \lesssim \| \langle x\rangle^{(p+1)c\sigma}\alpha\|_{L^{\frac{2d(p+2)}{4-dp}}} \|\langle x\rangle^{-c\sigma}u\|_{L_{t,x}^{p+2}}^{p+1}  \lesssim |v|^{-\frac{p+1}{p+2}}\|\varphi\|_{X^{1,\sigma}}^{p+1}. 
\end{align*}
Choosing $|v|$ sufficiently large we obtain $\Phi:Y\to Y$.

Using similar estimates, we find that for $u,v\in Y$,
\begin{align*}
\|\langle x&\rangle^{-c\sigma}[\Phi(u)-\Phi(v)]\|_{L_{t,x}^{p+2}} \\
& \lesssim \|\langle x\rangle^{(p+1)c\sigma}\alpha\|_{L^{\frac{2d(p+2)}{4-dp}}}\{\|\langle x\rangle^{-c\sigma}u\|_{L_{t,x}^{p+2}}^p + \|\langle x\rangle^{-c\sigma}v\|_{L_{t,x}^{p+2}}^p\}\\
& \quad\quad \times \|\langle x\rangle^{-c\sigma}[u-v]\|_{L_{t,x}^{p+2}} \\
& \lesssim |v|^{-\frac{p}{p+1}}\|\langle x\rangle^{-c\sigma}[u-v]\|_{L_{t,x}^{p+2}},
\end{align*}
so that $\Phi$ is a contraction provided $|v|$ is sufficiently large.  We therefore obtain the global solution $u$ to \eqref{nls} satisfying \eqref{WSTB}.

The $L^2$ convergence $e^{-it\Delta} u(t)\to u_-$ as $t\to-\infty$ follows from the estimates above.  It remains to prove the existence of the $L^2$ scattering state $u_+$.  In fact, by the estimates above, we have
\begin{align*}
\|e^{-it\Delta}u(t)-e^{-is\Delta}u(s)\|_{L^2} & \lesssim \|\alpha |u|^p u\|_{L_t^{q'}L_x^{r'}((s,t)\times\R^d)} \\
& \lesssim \|\langle x\rangle^{-c\sigma}u\|_{L_{t,x}^{p+2}((s,t)\times\R^d)}^{p+1} \to 0
\end{align*}
as $s,t\to\infty$, which yields the result. \end{proof}

\section{The Inverse Problem}

Theorems~\ref{T:intercritical}, \ref{T:subcritical}, and \ref{T:subcritical2} show that under suitable assumptions on $(p,\alpha)$ we can define final states $u_+$ corresponding to data $u_-$ via the solution to \eqref{nls}. We denote the scattering map sending $u_-$ to $u_+$ by
\[
S=S_{p,\alpha}:A\to \begin{cases} H^1 & \text{in Theorem~\ref{T:intercritical}}, \\ L^2 & \text{in Theorems~\ref{T:subcritical}~and~\ref{T:subcritical2}.}\end{cases}
\]
The proof of Theorem~\ref{T:intercritical} shows that we may take
\[
A=\{\varphi\in H^1:\|\varphi\|_{H^1}<\eta\}\qtq{in Theorem~\ref{T:intercritical},}
\]
with $\eta=\eta(\alpha)$ sufficiently small. Similarly, we may take
\[
A=\{\varphi\in L^2:\|\varphi\|_{L^2}<\eta\}\qtq{in Theorem~\ref{T:subcritical},}
\]
with $\eta=\eta(\alpha)$ sufficiently small. Finally, choosing $\sigma>0$ satisfying \eqref{sigma}, the proof of Theorem~\ref{T:subcritical2} shows that we may take
\[
A=\bigcup_{M>0} \{e^{iv\cdot x}\varphi:\|\varphi\|_{X^{1,\sigma}}\leq M,\quad |v|>CM^{p+2}\} \qtq{in Theorem~\ref{T:subcritical2},}
\]
where $C=C(\alpha)$ is sufficiently large.

In all cases, we have the following implicit formula for $S$:
\[
Su_- = u_- - i\int_\R e^{-it\Delta} \alpha|u|^p u(t)\,dt,
\]
where $u$ is the solution to \eqref{nls} that scatters backward in time to $u_-$.  We wish to show that knowledge of $S$ is sufficient to determine the nonlinearity in \eqref{nls}.

We first consider the case of Theorem~\ref{T:intercritical} and Theorem~\ref{T:subcritical} and prove a result similar to the one appearing in \cite{Strauss}. 

\begin{theorem}\label{T:recovery1} Let $(d,p,\alpha)$ satisfy the assumptions of Theorem~\ref{T:intercritical} or Theorem~\ref{T:subcritical} Let $S$ denote the corresponding scattering map.  Let
\[
\varphi\in\begin{cases} H^1 & \text{in the case of Theorem~\ref{T:intercritical}}, \\ L^2 & \text{in the case of Theorem~\ref{T:subcritical},}\end{cases}
\]
and let $\psi\in L^2$. Then 
\[
\lim_{\eps\to 0}i\eps^{-(p+1)}\langle (S-I)(\eps\varphi),\psi\rangle = \int_\R \langle\alpha |e^{it\Delta}\varphi|^{p}e^{it\Delta}\varphi,e^{it\Delta}\psi\rangle\,dt. 
\]
\end{theorem}

\begin{proof} Given $\eps>0$ sufficiently small, we set $u_-=\eps\varphi\in A$ and let $u$ be the corresponding solution to \eqref{nls} constructed in Theorem~\ref{T:intercritical} or Theorem~\ref{T:subcritical}.  We write
\begin{align}
i\langle (S-I)(\eps\varphi),\psi\rangle & = \int_\R \langle \alpha |u|^p u,e^{it\Delta}\psi\rangle\,dt \nonumber\\
& = \eps^{p+1}\int_\R \langle\alpha |e^{it\Delta}\varphi|^{p}e^{it\Delta}\varphi,e^{it\Delta}\psi\rangle\,dt \nonumber\\
& \quad + \int \langle\alpha[|u|^p u - |e^{it\Delta}\eps\varphi|^p e^{it\Delta}\eps\varphi],e^{it\Delta}\psi\rangle\,dt. \label{Err1}
\end{align}
To complete the proof, we will show that $|\eqref{Err1}|=o(\eps^{p+1})$.

We begin by using the Duhamel formula (cf. \eqref{Phi}) to obtain the estimate
\begin{align*}
|\eqref{Err1}| & \lesssim \| \alpha e^{it\Delta}\psi\bigl[|u|^p + |e^{it\Delta}\eps\varphi|^p\bigr]\,\bigl[u-e^{it\Delta}\eps\varphi\bigr]\|_{L_{t,x}^1} \\
& \lesssim \biggl\|\alpha e^{it\Delta}\eps\psi \bigl[|u|^p + |e^{it\Delta}\eps\varphi|^p\bigr]\biggl[\int_{-\infty}^t e^{i(t-s)\Delta}\alpha|u|^p u(s)\,ds\biggr]\biggr\|_{L_{t,x}^1}
\end{align*}

We first consider the setting of Theorem~\ref{T:intercritical}.  Using the estimates appearing in the proof of that theorem, we apply H\"older's inequality, Strichartz, and \eqref{STB1} to obtain
\begin{align*}
|\eqref{Err1}| & \lesssim \|\alpha\|_{L^\infty} \|e^{it\Delta}\psi[|u|^p +|e^{it\Delta}\eps\varphi|^p]\|_{L_t^{q'}L_x^{r'}}\biggl\|\int_0^t e^{i(t-s)\Delta}\alpha|u|^p u(s)\,ds\biggr\|_{L_t^q L_x^r} \\
& \lesssim \|e^{it\Delta}\psi\|_{L_t^q L_x^r}\bigl[\|u\|_{L_t^q L_x^{r_c}}^p+\|e^{it\Delta}\eps\varphi\|_{L_t^q L_x^{r_c}}^p\bigr]\|\alpha |u|^p u\|_{L_t^{q'}L_x^{r'}} \\
& \lesssim \|\psi\|_{L^2}\|\eps\varphi\|_{H^1}^{2p+1} \lesssim \eps^{2p+1},
\end{align*}
which is acceptable. 

We next consider the setting of Theorem~\ref{T:subcritical}.  Applying the estimates used to prove that theorem and \eqref{STB2}, we obtain
\begin{align*}
|\eqref{Err1}| & \lesssim \|\alpha e^{it\Delta}\psi[|u|^p + |e^{it\Delta}\eps\varphi|^p]\|_{L_t^{q'}L_x^{r'}}\biggl\| \int_0^t e^{i(t-s)\Delta}\alpha|u|^p u(s)\,ds\biggr\|_{L_t^q L_x^r} \\
& \lesssim \|\alpha\|_{L^{\frac{2d}{4-dp}}}\|e^{it\Delta}\psi\|_{L_t^q L_x^r}\bigl[\|u\|_{L_t^q L_x^r}^p +\|e^{it\Delta}\varphi\|_{L_t^q L_x^r}^p\bigr]\|\alpha |u|^p u\|_{L_t^{q'}L_x^{r'}} \\
& \lesssim \|\psi\|_{L^2}\|\eps\varphi\|_{L^2}^{2p+1}\lesssim \eps^{2p+1}, 
\end{align*}
which is acceptable. \end{proof}

\begin{corollary}\label{C1} Let $d\geq 1$ and suppose $(p,\alpha)$ and $(\tilde p,\tilde\alpha)$ satisfy the assumptions of Theorem~\ref{T:intercritical} or Theorem~\ref{T:subcritical}.  Let $S:A\to L^2$ and $\tilde S:\tilde A\to L^2$ denote the corresponding scattering maps. 

If $S(f)=\tilde S(f)$ for all $f\in A\cap \tilde A$, then $p=\tilde p$ and $\alpha=\tilde \alpha$.  
\end{corollary}

\begin{proof} Fix $\varphi\in \mathcal{S}$.  The proof of Theorem~\ref{T:recovery1} shows that
\[
\frac{\langle (S-I)(2\eps\varphi),\varphi\rangle}{\langle (S-I)(\eps\varphi),\varphi\rangle} = \frac{2^{p+1}\eps^{p+1}C+\mathcal{O}(\eps^{2p+1})}{\eps^{p+1}C+\mathcal{O}(\eps^{2p+1})} \to 2^{p+1} \qtq{as}\eps\to 0,
\]
where
\[
C=C(\alpha,\varphi,p):=\iint \alpha(x)|e^{it\Delta}\varphi|^{p+2}\,dx\,dt. 
\]
Similarly, 
\[
\frac{\langle (\tilde S-I)(2\eps\varphi),\varphi\rangle}{\langle (\tilde S-I)(\eps\varphi),\varphi\rangle}\to 2^{\tilde p+1}\qtq{as}\eps\to 0.
\]
Thus if $S=\tilde S$, we first obtain $p=\tilde p$.

Applying Theorem~\ref{T:recovery1}, we further obtain
\begin{equation}\label{TEB}
\iint \alpha(x)|e^{it\Delta}\varphi|^{p+2}\,dx\,dt = \iint \tilde \alpha(x) |e^{it\Delta}\varphi|^{p+2}\,dx\,dt \qtq{for all}\varphi\in\mathcal{S}.
\end{equation}
It therefore suffices to prove that if
\begin{equation}\label{TEB2}
\iint \alpha(x)|e^{it\Delta}\varphi(x)|^{p+2}\,dx\,dt = 0\qtq{for all}\varphi\in\mathcal{S},
\end{equation}
then $\alpha\equiv 0$.  To prove this, we utilize an argument appearing in \cite{ChenMurphy} (see also \cite{KMV}).

First, given $\varphi\in\mathcal{S}$, we define
\[
K_\varphi(x) = \int_{\R} |e^{it\Delta}\varphi(x)|^{p+2}\,dt
\]
and claim that $K_\varphi\in L^2(\R^d)$.  To prove this, we use Minkowski's integral inequality, Sobolev embedding, and the dispersive estimate to obtain the following:
\begin{align*}
\|\,\|e^{it\Delta}\varphi\|_{L_t^{p+2}}^{p+2}\|_{L_x^2} & \lesssim \|e^{it\Delta}\varphi\|_{L_x^{2(p+2)}L_t^{p+2}}^{p+2} \\
& \lesssim \|e^{it\Delta}\varphi\|_{L_t^{p+2} L_x^{2(p+2)}}^{p+2} \\
& \lesssim_\varphi \|\langle t\rangle^{-[\frac{d}{2}-\frac{d}{2(p+2)}]}\|_{L_t^{p+2}}^{p+2}\lesssim_\varphi 1
\end{align*}
provided $p>\max\{\frac{2}{d}-1,0\}$. 

We now specialize to the case
\[
\varphi(x) = \exp\{-\tfrac{|x|^2}{4}\},\qtq{so that} e^{it\Delta}\varphi(x) = \bigl(1+it)^{-\frac{d}{2}}\exp\{-\tfrac{|x|^2}{4(1+it)}\}
\]
(see \cite{Visan}).  In particular, we have
\[
K_\varphi(x) = \int_\R (1+t^2)^{-\frac{d(p+2)}{4}}\exp\{-\tfrac{(p+2)|x|^2}{4(1+t^2)}\}\,dt,
\]
and so by translation invariance for the linear Schr\"odinger equation, \eqref{TEB2} implies
\[
\int \alpha(x) K_\varphi(x-x_0)\,dx = 0 \qtq{for all}x_0\in \R^d. 
\]
To see that this implies $\alpha\equiv 0$, it therefore suffices to verify that $\hat K_\varphi\neq 0$ almost everywhere.  In fact, for any $\xi\neq 0$, we can compute $\hat K_\varphi(\xi)$ as a Gaussian integral:
\begin{align*}
\hat K_\varphi(\xi) & = (2\pi)^{-\frac{d}{2}}\int_\R (1+t^2)^{-\frac{d(p+2)}{4}} \int_{\R^d} \exp\{-ix\cdot\xi-\tfrac{p+2}{4(1+t^2)}|x|^2\}\,dx\,dt \\
& = c_{d,p}\int_\R (1+t^2)^{-\frac{dp}{4}}\exp\{-\tfrac{(1+t^2)|\xi|^2}{p+2}\}\,dt. 
\end{align*}
As $\hat K_\varphi(\xi)$ is the integral of a positive function, we conclude that $\hat K_\varphi(\xi)>0$ for all $\xi\neq 0$. \end{proof}

We next consider the case of Theorem~\ref{T:subcritical2} and prove a result similar to the one appearing in \cite{Watanabe}.  We recall the spaces $X^{a,s}$ defined in \eqref{Xas}, \eqref{Xs}.

\begin{theorem}\label{T:recovery2} Let $d\geq 3$.  Suppose $(p,\alpha)$ satisfy the assumptions of Theorem~\ref{T:subcritical2} and choose $\sigma$ satisfying \eqref{sigma}.  Assume additionally that 
\begin{equation}\label{alpha-int}
|\alpha(x)|\lesssim \langle x\rangle^{-s}\qtq{for some}s\in(1,\tfrac{d}{2}).
\end{equation}
Let $S:A\to L^2$ denote the corresponding scattering map, and let $\varphi,\psi\in X^{\frac{d}{2}+,\sigma}.$  

For any $\theta\in\mathbb{S}^{d-1}$, 
\[
\lim_{\rho\to\infty} i\rho\langle(S-I)(e^{i\rho\theta\cdot x}\varphi),e^{i\rho\theta\cdot x}\psi\rangle = \int_\R \langle \alpha(\cdot + 2t\theta) |\varphi|^p \varphi, \psi\rangle\,dt. 
\]
\end{theorem}

\begin{proof} We fix $\theta\in\mathbb{S}^{d-1}$, let $\rho\gg 1$, and set $v=\rho\theta$. Let $u_-=e^{iv\cdot x}\varphi\in A$, and let $u$ be the corresponding solution to \eqref{nls} constructed in Theorem~\ref{T:subcritical2}.  We begin by writing
\begin{align}
i\langle (S-&I)(e^{iv\cdot x}\varphi),e^{iv\cdot x}\psi\rangle \nonumber \\
& = \int_\R \langle \alpha |u|^p u, e^{it\Delta}e^{iv\cdot x}\psi\rangle \,dt \nonumber \\
& = \int_\R \langle\alpha |e^{it\Delta}e^{iv\cdot x}\varphi|^p e^{it\Delta}e^{iv\cdot x}\varphi,e^{it\Delta}e^{iv\cdot x}\psi\rangle \,dt \label{WMain1} \\
& \quad + \int_\R \langle \alpha[|u|^p u-|e^{it\Delta}e^{iv\cdot x}\varphi|^p e^{it\Delta}e^{iv\cdot x}\varphi],e^{it\Delta}e^{iv\cdot x}\psi\rangle\,dt. \label{WErr1} 
\end{align}
We will extract the main term from \eqref{WMain1} and estimate \eqref{WErr1} as an error term. 

Using \eqref{boost} and a change of variables, we first obtain
\begin{align}
\eqref{WMain1} & = \int_\R \langle \alpha(\cdot)\,|e^{it\Delta}\varphi(\cdot-2\rho\theta t)|^p e^{it\Delta}\varphi(\cdot-2\rho\theta t),e^{it\Delta}\psi(\cdot-2 \rho\theta t)\rangle \,dt \nonumber \\ 
& = \tfrac{1}{\rho}\int_\R \langle \alpha(\cdot+2\theta t)|e^{i\frac{t}{\rho}\Delta}\varphi|^p e^{i\frac{t}{\rho}\Delta}\varphi,e^{i\frac{t}{\rho}\Delta}\psi\rangle\,dt. \nonumber
\end{align}
We now define 
\begin{align*}
h_\rho(t) &= \langle \alpha(\cdot+2\theta t)|e^{i\frac{t}{\rho}\Delta}\varphi|^p e^{i\frac{t}{\rho}\Delta}\varphi,e^{i\frac{t}{\rho}\Delta}\psi\rangle, \\
\ell(t) &= \langle \alpha(\cdot + 2\theta t)|\varphi|^p \varphi,\psi\rangle.
\end{align*}
We will prove that for all $t\in\R$, we have
\begin{align}
&\lim_{\rho\to\infty}h_\rho(t) = \ell(t),\qtq{and} \label{hDCT1} \\
&|h_\rho(t)|\lesssim \langle t\rangle^{-s}\in L_t^1. \label{hDCT2}
\end{align}

To this end, first observe that
\begin{align}
|h_\rho(t)-\ell(t)| & \leq |\langle \alpha(\cdot+2\theta t)|e^{i\frac{t}{\rho}\Delta}\varphi|^p e^{i\frac{t}{\rho}\Delta}\varphi-|\varphi|^p \varphi,e^{i\frac{t}{\rho}\Delta}\psi\rangle|\label{Main1-Error1} \\
& \quad + |\langle \alpha(\cdot+2\theta t)|\varphi|^p \varphi,e^{i\frac{t}{\rho}\Delta}\psi - \psi\rangle|.\label{Main1-Error2}
\end{align}
To estimate these terms, we use the pointwise bound
\[
|e^{-i\tau|\xi|^2}-1|\leq |\tau|^{\frac12} |\xi|.
\]
In particular, using $H^{\frac{d}{2}+}\hookrightarrow L^{2(p+1)}$, 
\begin{align*}
\eqref{Main1-Error2} & \leq (\tfrac{|t|}{\rho})^{\frac12}  \|\alpha\|_{L^\infty} \|\varphi\|_{L^{2(p+1)}}^{p+1} \|\nabla\psi\|_{L^2} \to 0 \qtq{as}\rho\to\infty.
\end{align*}
Similarly, using Sobolev embedding to control the free evolution in $L^\infty$, 
\begin{align*}
\eqref{Main1-Error1} & \leq \|\alpha\|_{L^\infty}\{\| e^{i\frac{t}{\rho}\Delta}\varphi\|_{L^\infty}^p+\|\varphi\|_{L^\infty}^p\}\|e^{i\frac{t}{\rho}\Delta}\varphi-\varphi\|_{L^2}\|e^{i\frac{t}{\rho}\Delta}\psi\|_{L^2} \\
& \lesssim (\tfrac{|t|}{\rho})^{\frac12}\|\alpha\|_{L^\infty}\|\varphi\|_{H^{\frac{d}{2}+}}^p\|\nabla\varphi\|_{L^2}\|\psi\|_{L^2} \to 0 \qtq{as}\rho\to\infty. 
\end{align*}
This proves \eqref{hDCT1}.

Next, we use \eqref{boost}, H\"older's inequality, Sobolev embedding, and Proposition~\ref{P:Enss} (in the form \eqref{PETS}), to obtain
\begin{align*}
|h_\rho(t)| &= |\langle \alpha \,|e^{i\frac{t}{\rho}\Delta} e^{i\theta\cdot x}\varphi|^p e^{i\frac{t}{\rho}\Delta} e^{i\theta\cdot x}\varphi, e^{i\frac{t}{\rho}\Delta}e^{i\theta\cdot x}\psi\rangle| \\
& \lesssim \|e^{i\frac{t}{\rho}\Delta}\varphi\|_{L^{2(p+1)}}^{p+1}\| \alpha e^{i\frac{t}{\rho}\Delta}e^{i\theta\cdot x}\psi \|_{L^2} \\
& \lesssim \|\varphi\|_{H^{\frac{d}{2}+}}^{p+1}\|\alpha(\cdot+2\tfrac{t}{\rho}v)e^{i\frac{t}{\rho}\Delta}\psi\|_{L^2}  \lesssim \langle t\rangle^{-s},
\end{align*}
which proves \eqref{hDCT2}.

By the dominated convergence theorem, we therefore obtain 
\[
\lim_{\rho\to\infty} \rho\cdot\eqref{WMain1} = \int_\R \langle \alpha(\cdot+2\theta t)|\varphi|^p\varphi,\psi\rangle\,dt,
\]
which yields the main term. 

To complete the proof, it remains to prove that
\[
|\eqref{WErr1}| = o(\rho^{-1})\qtq{as}\rho\to\infty.
\]
We begin with the estimate
\begin{align*}
|\eqref{WErr1}| & \lesssim \|\alpha [e^{it\Delta} e^{iv\cdot x}\psi][|u|^p + |e^{it\Delta}e^{iv\cdot x}\varphi|^p]\,[u-e^{it\Delta}e^{iv\cdot x}\varphi]\|_{L_{t,x}^1} \\
& \lesssim \biggl\| \alpha [e^{it\Delta} e^{iv\cdot x}\psi][|u|^p + |e^{it\Delta}e^{iv\cdot x}\varphi|^p]\,\biggl[\int_{-\infty}^t e^{i(t-s)\Delta}\alpha |u|^p u(s)\,ds\biggr]\biggr\|_{L_{t,x}^1}.
\end{align*}

We utilize Strichartz and the estimates appearing in the proof of Theorem~\ref{T:subcritical2} to obtain
\begin{align*}
|\eqref{WErr1}| & \lesssim \|\alpha[e^{it\Delta}e^{iv\cdot x}\psi][|u|^p +|e^{it\Delta} e^{iv\cdot x}\varphi|^p]\|_{L_t^{q'} L_x^{r'}} \| \alpha |u|^p u\|_{L_t^{q'} L_x^{r'}} \\
& \lesssim \|\langle x\rangle^{(p+1)c\sigma}\alpha\|_{L^{\frac{2d(p+2)}{4-dp}}}^2 \|\langle x\rangle^{-c\sigma} e^{it\Delta} e^{iv\cdot x}\psi\|_{L_{t,x}^{p+2}} \\
& \quad \times \bigl[\|\langle x\rangle^{-c\sigma}u\|_{L_{t,x}^{p+2}}^p + \| \langle x\rangle^{-c\sigma} e^{it\Delta}e^{iv\cdot x}\varphi\|_{L_{t,x}^{p+2}}^p\bigr] \|\langle x\rangle^{-c\sigma} u\|_{L_{t,x}^{p+2}}^{p+1} \\
& \lesssim \rho^{-2\frac{p+1}{p+2}} = o(\rho^{-1})\qtq{as}\rho\to\infty,
\end{align*}
as was needed to show. \end{proof}

\begin{corollary}\label{C2} Let $d\geq 3$ and suppose $(p,\alpha), (\tilde p,\tilde \alpha)$ satisfy the assumptions of Theorem~\ref{T:subcritical2} and Theorem~\ref{T:recovery2}.  Let $S:A\to L^2$ and $\tilde S:\tilde A\to L^2$ denote the corresponding scattering maps.

If $S(f)=\tilde S(f)$ for all $f\in A\cap\tilde A$, then $p=\tilde p$ and $\alpha=\tilde \alpha$. 
\end{corollary}

\begin{proof} Fix $\varphi\in\mathcal{S}$ and $\theta\in\mathbb{S}^{d-1}$. Then Theorem~\ref{T:recovery2} implies
\[
\frac{\langle (S-I)(2e^{i\rho\theta\cdot x}\varphi),e^{i\rho\theta\cdot x}\varphi\rangle}{\langle (S-I)(e^{i\rho\theta\cdot x}\varphi),e^{i\rho\theta\cdot x}\varphi\rangle}\to 2^{p+1}\qtq{as}\rho\to\infty. 
\]
Thus if $S=\tilde S$, we first obtain $p=\tilde p$.  

Applying Theorem~\ref{T:recovery2} once again, we obtain
\begin{equation}\label{WAF1}
\int_\R\int_{\R^d}  \alpha(x+2\theta t)|\varphi(x)|^{p+2}\,dx\,dt = \int_\R\int_{\R^d} \tilde\alpha(x+2\theta t)|\varphi(x)|^{p+2}\,dx\,dt
\end{equation}
for all $\theta\in\mathbb{S}^{d-1}$ and $\varphi\in X^{\frac{d}{2}+,\sigma}$.

We now fix $\theta\in\mathbb{S}^{d-1}$ and $y\in\R^d$.  We then choose a nonnegative, compactly supported $\varphi\in L^1$ with $\int \varphi = 1$ and set $\varphi_n(x) = n^d \varphi(nx)$.  By \eqref{WAF1}, we have
\begin{equation}\label{Weqn}
\iint \alpha(x+2\theta t)\varphi_n(x-y)\,dx\,dt = \iint\tilde\alpha(x+2\theta t)\varphi_n(x-y)\,dx\,dt \qtq{for all $n$.} 
\end{equation}

Now consider the functions
\[
g_n(t) = \int_{\R^d} \alpha(x+2\theta t)\varphi_n(x-y)\,dx. 
\]
By approximate identity arguments, we have that 
\[
g_n(t)\to \alpha(y+2\theta t)\qtq{as} n\to\infty\qtq{for all}t\in\R.
\]
Furthermore, recalling \eqref{alpha-int} and noting that $|x-y|\lesssim 1$ on the support of $\varphi_n(x-y)$, we have
\[
|g_n(t)|  \lesssim \int \langle x+2\theta t\rangle^{-s} \varphi_n(x-y)\,dx \lesssim \int h(t)\varphi_n(x-y)\,dx \lesssim h(t),
\]
where $h\in L_t^1$ is defined by
\[
h(t):=\begin{cases} 1 & |t|\lesssim |y| \\ \langle t\rangle^{-s} & |t|\gg |y|.\end{cases}
\]
Thus, by the dominated convergence theorem, we have
\[
\int_\R\int_{\R^d}\alpha(x+2\theta t)\varphi_n(x-y)\,dx\,dt \to \int_\R \alpha(y+2\theta t)\,dt \qtq{as}n\to\infty. 
\]

Arguing similarly for $\tilde\alpha$ and recalling \eqref{Weqn}, we deduce
\[
\int_\R \alpha(y+\theta t)\,dt = \int_\R\tilde \alpha(y+\theta t)\,dt.
\]
As $\theta\in\mathbb{S}^{d-1}$ and $y\in\R^d$ were arbitrary, the fact that $\alpha=\tilde\alpha$ now follows from the injectivity of the X-ray transform (see e.g. \cite[Chapter I]{Helgason}). \end{proof}

\subsection{Challenges in the mass-supercritical regime}  The approach taken in Theorem~\ref{T:subcritical2} and Theorem~\ref{T:recovery2} is to formulate the scattering problem as a small-data problem, capitalizing on the fact that highly boosted data become small in weighted spaces (a consequence of Proposition~\ref{P:Enss}).  This construction guarantees that the corresponding nonlinear solutions inherit the weighted estimates enjoyed by the boosted linear solutions.  Such estimates then play an essential role in the proof of Theorem~\ref{T:recovery2}, particularly in the estimation of the error term \eqref{WErr1}. 

Extending this approach to the mass-supercritical regime seems to lead to some significant difficulties.  Indeed, in this setting the small-data contraction mapping argument to construct the scattering solutions requires some estimates on the derivatives of solutions; however, the derivatives of highly boosted data will become very large. Thus, while it seems possible to extend Theorem~\ref{T:subcritical2} and Theorem~\ref{T:recovery2} into the \emph{slightly} mass-supercritical regime, the full intercritical and energy-critical regime appear to be out of reach for now. 

In \cite{Watanabe}, Watanabe proceeded by imposing a positivity and repulsivity condition on the coefficient $\alpha$, which allowed for the use of Morawetz estimates to establish an intercritical scattering theory for \eqref{nls} for arbitrarily large $H^1$ data (including boosted data).  As in the proof of Theorem~\ref{T:recovery2}, the recovery problem subsequently required the estimation of an error term like \eqref{WErr1}.  The approach of \cite{Watanabe} was based on the intertwining property and an implicit formula for the wave operator $\Omega_-$; however, it appears that the formula for $\Omega_-$ in \cite[Lemma~3.2]{Watanabe} is missing a factor of $\Omega_-$ in the nonlinear term.  In the absence of this factor, one is ultimately faced with estimating only a \emph{linear} term, for which an estimate such as Proposition~\ref{P:Enss} is sufficient.  Restoring the missing factor of $\Omega_-$, one is instead led to a term involving the full (nonlinear) solution.  It then seems necessary to prove that even in this setting, the scattering solutions inherit the weighted estimates satisfied by the boosted linear solutions.  At present, the author is not aware of a method to obtain such estimates in the intercritical setting.  On the other hand, Theorem~\ref{T:subcritical2} and Theorem~\ref{T:recovery2} demonstrate that in the mass-critical and mass-subcritical regime, the scattering problem does admit a formulation as a small-data problem that is well-adapted to the approach found in \cite{Watanabe}.

\end{document}